\documentclass[a4paper,10pt,reqno]{amsart}
\usepackage{amssymb,amsmath,amsthm}
\usepackage[hidelinks]{hyperref}
\usepackage{multirow}
\usepackage{graphicx}
\author{\'Akos K. Matszangosz, Ferenc Sz\"oll\H{o}si}
\date{May 16, 2024. Preprint.}
\address{\'A.K.M.: HUN-REN Alfr\'ed R\'enyi Institute of Mathematics, Re\'altanoda utca 13--15, 1053 Budapest, Hungary}
\email{matszangosz.akos@gmail.com}
\address{F.Sz.: Department of Mathematical Sciences, Interdisciplinary Faculty of Science and Engineering, Shimane University, 1060 Nishikawatsucho, Matsue, Shimane, Japan}
\email{szollosi@riko.shimane-u.ac.jp}
\newtheorem{theorem}{Theorem}[section]
\newtheorem{lemma}[theorem]{Lemma}
\newtheorem{proposition}[theorem]{Proposition}

\theoremstyle{definition}
\newtheorem{defi}{Definition}
\newtheorem{example}{Example}
\newtheorem{remark}{Remark}
\title[A characterization of Hadamard matrices appearing in MUB triplets]{A characterization of complex Hadamard matrices appearing in families of MUB triplets}
\begin{document}
\begin{abstract}
It is shown that a normalized complex Hadamard matrix of order $6$ having three distinct columns, each containing at least one $-1$ entry necessarily belongs to the transposed Fourier family, or to the family of $2$-circulant complex Hadamard matrices. The proofs rely on solving polynomial system of equations by Gr\"obner basis techniques, and make use of a structure theorem concerning regular Hadamard matrices. As a consequence, members of these two families can be easily recognized in practice. In particular, one can identify complex Hadamard matrices appearing in known triplets of pairwise mutually unbiased bases in dimension $6$.
\end{abstract}
\maketitle
\section{Introduction and main results}
Let $\mathbb{C}^n$ denote the usual $n$-dimensional complex vector space equipped with the Hermitian inner product $\left\langle\cdot,\cdot\right\rangle$. A complex Hadamard matrix $H$ of order $n$ is an $n\times n$ complex matrix with unimodular entries, having pairwise complex orthogonal rows. Real and complex Hadamard matrices have been extensively studied because they are useful in applications in communication engineering and quantum information theory \cite{B}, \cite{TZ}. Two Hadamard matrices are called equivalent, if one can be transformed into the other by permuting its rows and columns, and multiplying its rows and columns by complex numbers of modulus $1$. The set of all complex Hadamard matrices equivalent to $H$ is called the equivalence class of $H$. Every equivalence class contains a normalized matrix whose first entries in its rows and columns are preset to $1$. Complex Hadamard matrices have been classified up to $n\leq 5$, and there are a number of known infinite, parametric families of order $6$, see \cite{BN}, \cite{D}, \cite{UH}, \cite{SZ1}, \cite{SZ2}.

Closely related to Hadamard matrices is the concept of mutually unbiased bases (MUBs) \cite{BONDAL}, \cite{GR}, \cite{MORA}, \cite{J}, \cite{L}. A pair of orthonormal bases $\{e_i\colon i\in\{1,2,\dots,n\}\}$ and $\{f_i\colon i\in\{1,2,\dots,n\}\}$ in $\mathbb{C}^n$ is called mutually unbiased, if $|\left\langle e_i,f_j\right\rangle|^2=1/n$ for every $i,j\in\{1,2,\dots,n\}$. By representing these bases as unitary matrices $E$ and $F$ in the standard basis, the transition matrix $T:=\sqrt{n}E^\ast F$ becomes a complex Hadamard matrix (here $E^\ast$ denotes the usual Hermitian adjoint). There are various generalizations of MUBs, such as mutually unbiased weighing matrices \cite{KH}.

While the maximum number $M(n)$ of pairwise mutually unbiased bases in $\mathbb{C}^n$ is known to be $M(p^r)=p^r+1$ whenever $n=p^r$ is a prime power, it is a long-standing open problem to determine the exact value of $M(6)$. Zauner conjectured \cite{Z} that $M(6)=3$, but this remains unresolved despite continuous efforts \cite{BW}, \cite{G}, \cite{HRZ}, \cite{MC}.

Supporting his conjecture, Zauner \cite[pp.~71--74]{Z} exhibits triplets of pairwise mutually unbiased bases using $2$-circulant transition matrices, and in particular establishes the lower bound $M(6)\geq 3$. The main ingredients of this construction are the transposed Fourier family \cite{TZ}:
\begin{equation}\label{Fab}\def\arraystretch{1.2}
F_6^T(a,b)=\left[
\begin{array}{ccc|rrr}
 1 & 1 & 1 & 1 & 1 & 1 \\
 1 & 1 & 1 & -1 & -1 & -1 \\
\hline
 1 & w & w^2 & a & a w & a w^2 \\
 1 & w & w^2 & -a & -a w & -a w^2 \\
\hline
 1 & w^2 & w & b & b w^2 & b w \\
 1 & w^2 & w & -b & -b w^2 & -b w \\
\end{array}
\right],\ \begin{cases}1+w+w^2=0,\\
|a|=|b|=1;
\end{cases}
\end{equation}
and the $2$-circulant family (see the top-right array on \cite[p.~923]{SZ1}):
\begin{equation}\label{X6alpha}\arraycolsep=4.5pt\def\arraystretch{1.4}
X_6(\beta,\gamma,\varepsilon,\varphi)=\left[\begin{array}{cccccc}
1 & 1 & 1 & 1 & 1 & 1\\
 1 & -1 & -\frac{1}{\gamma  \varepsilon } & -\frac{1}{\beta  \varphi } & \frac{1}{\gamma  \varepsilon } & \frac{1}{\beta  \varphi } \\
 1 & -\frac{\varepsilon }{\beta } & -1 & -\frac{\varepsilon }{\gamma  \varphi } & \frac{\varepsilon }{\gamma  \varphi } & \frac{\varepsilon }{\beta } \\
 1 & -\frac{\varphi }{\gamma } & -\frac{\varphi }{\beta  \varepsilon } & -1 & \frac{\varphi }{\gamma } & \frac{\varphi }{\beta  \varepsilon } \\
 1 & \frac{\varepsilon }{\beta } & \frac{\varphi }{\beta  \varepsilon } & \frac{1}{\beta  \varphi } & \frac{1}{\beta  \gamma } & \frac{\gamma }{\beta ^2} \\
 1 & \frac{\varphi }{\gamma } & \frac{1}{\gamma  \varepsilon } & \frac{\varepsilon }{\gamma  \varphi } & \frac{\beta }{\gamma ^2} & \frac{1}{\beta  \gamma } \\
\end{array}\right],\ |\beta|=|\gamma|=|\varepsilon|=|\varphi|=1,
\end{equation}
further subject to the implicit equation
\begin{equation}\label{implicit}
\beta\gamma \varepsilon^2 + \beta \gamma \varphi + \beta^2 \varepsilon \varphi + \gamma \varepsilon \varphi + \beta \gamma^2 \varepsilon \varphi + \beta \gamma \varepsilon \varphi^2=0.
\end{equation}
Using these two families, Zauner's construction leads to a two-parameter family \cite{SZ1} of triplets of mutually unbiased bases, up to a suitably defined equivalence relation.

Thus, the following questions arise naturally: given a triplet of mutually unbiased bases, represented by basis transformation matrices, how can we decide whether or not they come from Zauner's construction? In particular, given a transition matrix $T$, how can we ascertain the existence of a unimodular quadruple $(\beta^\ast,\gamma^\ast,\varepsilon^\ast,\varphi^\ast)$ subject to \eqref{implicit}, resulting in a matrix $X_6(\beta^\ast,\gamma^\ast,\varepsilon^\ast,\varphi^\ast)$ equivalent to $T$?

A common feature of the families shown in \eqref{Fab} and \eqref{X6alpha} is that their members have a normalized form containing three $-1$ entries. It turns out that they are characterized by this property. The main result of this paper is the following.
\begin{theorem}\label{T1}
Let $H$ be a normalized complex Hadamard matrix of order $6$, having three distinct columns, each containing at least one $-1$ entry. Then $H$ belongs to the transposed Fourier family \eqref{Fab}, or to the $2$-circulant family \eqref{X6alpha}, up to equivalence.
\end{theorem}
\begin{remark}\label{rem1}
The two families are known to intersect at the matrix $F_6^T(w,w)$, where $1+w+w^2=0$. Further, it is easy to see that the family $F_6^T(a,b)$ contains the same set of matrices irrespective of the actual value of the primitive cubic root $w$.
\end{remark}
The proof of Theorem~\ref{T1} consists of three parts, depending on the number of rows of $H$ containing a $-1$ entry. In Section~\ref{S1COL} we consider the case when there exists a single row containing three $-1$ entries, and show that in this case $H$ is a member of the transposed Fourier family. Then, after some preliminaries in Sections~\ref{SECPRE1} and \ref{SECPRE2}, we consider the case of three rows in Section~\ref{S2COLS}, and show that this case leads to the $2$-circulant family. Finally, we deal with the case of two rows in Section~\ref{S3COLS}. Our proofs rely on solving polynomial system of equations by Gr\"obner basis techniques (see \cite{BJ}, \cite{FA}, \cite{LAZ}, \cite{SHI} and the references therein), and make use of an elegant structure theorem concerning regular (in the sense of \cite{B}) complex Hadamard matrices. As a consequence of our results, recognizing whether or not a given $6\times 6$ Hadamard matrix belongs to either of these two families simply boils down to first normalizing the matrix and then having a look at the pattern of $-1$ entries in it. This provides a convenient characterization of these two families, see Theorems~\ref{F6CHAR} and \ref{T54}.%
%
\section{The case of a single row: The transposed Fourier family}\label{S1COL}
In this section we consider $6\times 6$ complex Hadamard matrices having a submatrix:
\begin{equation}\label{F6FORM}\left[\begin{array}{rrrr}
1 & 1 & 1 & 1\\
1 & -1 & -1 & -1\\
\end{array}\right].\end{equation}
That is, the normalized matrix has three distinct columns, each containing a $-1$ entry, which are furthermore located in the very same row.

We will frequently use the following Lemmata without any further comments.
\begin{lemma}
Let $a$, $b$, $c$ be complex unimodular numbers. Then, if $a+b+c=0$, then $b=aw$ and $c=aw^2$; or $b=aw^2$ and $c=aw$, where $1+w+w^2=0$.
\end{lemma}
\begin{lemma}\label{L22}
Let $a$, $b$, $c$, $d$ be complex unimodular numbers. Then, if $a+b+c+d=0$, then $(a+d)(b+d)(c+d)=0$. In particular, $d=-a$ and $c=-b$; or $d=-b$ and $c=-a$; or $d=-c$ and $b=-a$.
\end{lemma}
When $x,y\in\mathbb{C}^n$ are unimodular, then the Hermitian inner product becomes $\left\langle x,y\right\rangle=\sum_{k=1}^n x_{k}/y_{k}$. For completeness, we record the following folklore result.
\begin{proposition}\label{P21}
Let $H$ be a normalized complex Hadamard matrix of order $6$, containing three $-1$ entries in one of its rows. Then, up to equivalence, $H$ is a member of the transposed Fourier family, shown in \eqref{Fab}.
\end{proposition}
\begin{proof}
Let us denote the rows of $H$ by $h_i$, $i\in\{1,2,3,4,5,6\}$. Up to permutations, we may assume that the $-1$ entries are in the second row of $H$. Since the rows of $H$ are pairwise complex orthogonal, we have $\left\langle h_1,h_2\right\rangle=0$. Consequently the remaining entries in the second row are necessarily $+1$, and we have a matrix of the following form:
	\[\left[
	\begin{array}{rrrrrr}
		1 & 1 & 1 & 1 & 1 & 1 \\
		1 & 1 & 1 & -1 & -1 & -1 \\
		1 & w & x & a & y & z \\
	\end{array}
	\right].\]
By considering $\left\langle h_1,h_3\right\rangle=\left\langle h_2,h_3\right\rangle=0$, it follows that $1+w+x=0$, and hence $w$ and $x$ are primitive complex third root of unity, such that $x=w^2$. This property must hold for any subsequent rows of $H$, therefore, using orthogonality of the first three columns, it has the following form, up to a permutation of the rows:
	\[\left[
\begin{array}{cccrrr}
1 & 1 & 1 & 1 & 1 & 1 \\
1 & 1 & 1 & -1 & -1 & -1 \\
1 & w & w^2 & a & y & z \\
1 & w & w^2 &  &  &  \\
1 & w^2 & w & b &  &  \\
1 & w^2 & w &  &  &  \\
\end{array}
\right]\quad \leadsto\quad \left[
\begin{array}{cccrrr}
	1 & 1 & 1 & 1 & 1 & 1 \\
	1 & 1 & 1 & -1 & -1 & -1 \\
	1 & w & w^2 & a & aw & aw^2 \\
	1 & w & w^2 & -a & -aw & -aw^2 \\
	1 & w^2 & w & b & bw^2 & bw \\
	1 & w^2 & w & -b & -bw^2 & -bw \\
\end{array}
\right].\]
Similarly, since $a+y+z=0$, we see that $y/a$ and $z/a$ must be primitive complex third roots of unity. Up to permuting the last two columns, we may assume that $y=aw$ and $z=aw^2$. Subsequently, $\left\langle h_3,h_4\right\rangle=0$ implies that the unspecified entries in the fourth row must be the negative of those in the third. Considering $\left\langle h_3,h_5\right\rangle=0$, and using analogous arguments, we see that the remaining entries in the fifth row must be $bw^2$ and $bw$. Finally, $\left\langle h_5,h_6\right\rangle=0$ implies that the remaining entries in the last row are the negative of those in the fifth, and we arrive at the family $F_6^T(a,b)$.
\end{proof}
The following is a convenient characterization of this family.
\begin{theorem}\label{F6CHAR}
Let $H$ be a complex Hadamard matrix of order $6$. Then the following are equivalent:
\begin{enumerate}
\item[($a$)] every normalized matrix in the equivalence class of $H$ contains a (not necessarily contiguous) submatrix of the form \eqref{F6FORM};
\item[($b$)] some normalized matrix in the equivalence class of $H$ contains a (not necessarily contiguous) submatrix of the form \eqref{F6FORM};
\item[($c$)] $H$ is a member of the transposed Fourier family \eqref{Fab}, up to equivalence.
\end{enumerate}
\end{theorem}
\begin{proof}
The direction $(a)\Rightarrow(b)$ is trivial, and $(b)\Rightarrow(c)$ follows from the statement of Proposition~\ref{P21}. Therefore, it remains to show that $(c)\Rightarrow (a)$.

As highlighted in formula \eqref{Fab}, $F_6^T(a,b)$ has a natural block partition into $6$ rank-$1$ submatrices of size $2\times 3$ each. Therefore, irrespective of how its rows and columns are permuted, its first row together with another one will span two not necessarily contiguous, disjoint rank-$1$ submatrices of size $2\times 3$. 

Now consider $H$, normalize it, and look at the first row and its pair containing these two rank-$1$ submatrices. Clearly, one of these two submatrices must be the constant $1$ matrix. Further, since these two rows are complex orthogonal, the other $2\times 3$ rank-$1$ submatrix must contain three $-1$ entries. In particular, the normalized form of $H$ contains a submatrix \eqref{F6FORM}.
\end{proof}
\section{Polynomial equations}\label{SECPRE1}
This section introduces polynomial equations in the matrix entries that are used in order to obtain the characterization described in Theorem~\ref{T1}. The notation we use here is particularly convenient for computer algebra. Let $H=[h_{uv}]$ be a $6\times 6$ matrix. At this point we think of the entries $h_{uv}$ as variables. We associate a polynomial to the pair of rows $\{h_x,h_y\}$ of $H$ capturing orthogonality ($x,y\in\{1,2,3,4,5,6\}$):
\begin{equation}\label{ort}
\mathcal{O}((x,y)):=\left(\prod_{v=1}^6 h_{yv}\right)\cdot\left(\sum_{v=1}^6\frac{h_{xv}}{h_{yv}}\right).
\end{equation}
\begin{lemma}\label{L0}
Let $H=[h_{uv}]$ be a complex Hadamard matrix of order $6$. Let $i$, $j\in\{1,2,3,4,5,6\}$ be distinct. Then $\mathcal{O}((i,j))=\mathcal{O}((j,i))=0$.
\end{lemma}
\begin{proof}
Since $H$ is a complex Hadamard matrix, $1=|h_{uv}|^2=h_{uv}\overline{h}_{uv}$, hence $\overline{h}_{uv}=1/h_{uv}$. Therefore formula \eqref{ort} is well-defined, and evaluating it at $(i,j)$ captures the inner product between the $i$th and $j$th row of $H$, which, of course, must be $0$.
\end{proof}
\begin{remark}\label{newremark}
To see that a polynomial $\mathcal{O}((1,2))$ and its ``conjugate'' $\mathcal{O}((2,1))$ carries different information in general, consider the $2\times 3$ matrix $\left[\begin{smallmatrix}1&1&1\\ 1&a&b\end{smallmatrix}\right]$ with orthogonal rows. Here, while the single polynomial $\mathcal{O}((2,1))=1+a+b=0$ leads to infinitely many complex solutions, considering it together with $\mathcal{O}((1,2))=a+b+ab=0$ yields a system of two polynomial equations resulting in only finitely many solutions.
\end{remark}
Lemma~\ref{L0} describes an algebraic relation between the entries of a $2\times 6$ submatrix of $H$ using the polynomial \eqref{ort}. It turns out, that it is possible to associate analogous polynomials to $3\times 4$ submatrices. Let $i$, $j$, $k\in\{1,2,3,4,5,6\}$ be pairwise distinct, and let $p$, $q$, $r$, $s\in\{1,2,3,4,5,6\}$ be also pairwise distinct. We will frequently consider the following Haagerup polynomial $\mathcal{H}$, associated to the $3\times 4$ submatrix $M$ of $H$, spanned by rows $i$, $j$, $k$ and columns $p$, $q$, $r$, and $s$. First, let us define the following shorthand notations
\[\mathcal{I}((x,y)):=\frac{h_{xp}}{h_{yp}}+\frac{h_{xq}}{h_{yq}}+\frac{h_{xr}}{h_{yr}}+\frac{h_{xs}}{h_{ys}},\qquad \mathcal{J}:=\prod_{\substack{u\in\{i,j,k\}\\ v\in\{p,q,r,s\}}}h_{uv}.\]
Here $(x,y)\in\{i,j,k\}$ are distinct, and we no longer signify the dependence on $(p,q,r,s)$. We remark that when $H$ is an actual complex Hadamard matrix, then the expression $\mathcal{I}((x,y))$ is well-defined. Further, $\mathcal{I}((x,y))=\overline{\mathcal{I}((y,x))}$, and its meaning is the complex inner product between the $x$th and $y$th row of $M$. Now the Haagerup polynomial associated to the submatrix $M$ reads:
\begin{multline}\label{POLYH}
\mathcal{H}((i,j,k),(p,q,r,s)):=\bigl[\mathcal{I}((i,j))\mathcal{I}((j,k))\mathcal{I}((k,i))-4+\mathcal{I}((i,j))\mathcal{I}((j,i))\\
+\mathcal{I}((j,k))\mathcal{I}((k,j))+\mathcal{I}((k,i))\mathcal{I}((i,k))\bigr]\mathcal{J}.
\end{multline}
It is easy to see that $\mathcal{H}$ is invariant up to cyclic permutations of the row indices $(i,j,k)$, and it is invariant up to any permutations of the column indices $(p,q,r,s)$. The following result is well-known, and was used extensively during the classification of $5\times 5$ and construction of $6\times 6$ complex Hadamard matrices.
\begin{lemma}[\cite{UH},\cite{MI},\cite{SZ2}]\label{L1}
Let $H=[h_{uv}]$ be a complex Hadamard matrix of order $6$. Let $i$, $j$, $k\in\{1,2,3,4,5,6\}$ be pairwise distinct, and let $p$, $q$, $r$, $s\in\{1,2,3,4,5,6\}$ be also pairwise distinct. Then \[\mathcal{H}((i,j,k),(p,q,r,s))=\mathcal{H}((k,j,i),(p,q,r,s))=0.\]
\end{lemma}
We remark that when $H$ is a complex Hadamard matrix, then the transposed matrix $H^T$ is also a complex Hadamard matrix. Therefore, polynomials analogous to \eqref{POLYH} can be obtained by considering $H^T$ (or the $4\times 3$ submatrices of $H$).
\begin{example}
As an application of Lemma~\ref{L1}, we show that the following unimodular matrix
\begin{equation}\label{W}
W=\left[
\begin{array}{cccccc}
	1 & 1 & 1 & 1 & 1 & 1 \\
	1 & \mathbf{i} & -1 & -\mathbf{i} & \frac{1-\mathbf{i}}{\sqrt{2}} & \frac{-1+\mathbf{i}}{\sqrt{2}} \\
	1 & -1 & \frac{-1+\mathbf{i}}{\sqrt{2}} & \frac{1+\mathbf{i}}{\sqrt{2}} & \frac{-1-\mathbf{i}}{\sqrt{2}} & \frac{1-\mathbf{i}}{\sqrt{2}} \\
	1 & \frac{2 \sqrt{2}+\mathbf{i}}{3} & \frac{-1-\mathbf{i}}{\sqrt{2}} & \frac{-4-\sqrt{2}+(4-\sqrt{2})\mathbf{i}}{6}  &
	\frac{-1+2\sqrt{2}\mathbf{i}}{3} & -\mathbf{i} \\
\end{array}
\right],\end{equation}
which has four pairwise orthogonal rows, cannot be part of any $6\times 6$ complex Hadamard matrix (here $\mathbf{i}\in\mathbb{C}$ denotes the complex imaginary unit). Indeed, assume that $W^T$ forms the first four columns of a complex Hadamard matrix $H$. We calculate $\mathcal{H}((1,2,3),(1,2,3,4))=-4(2\sqrt{2}+\mathbf{i})/9\neq 0$. This contradicts Lemma~\ref{L1}.
\end{example}
\section{Cancelling rows and regular Hadamard matrices}\label{SECPRE2}
In this section we loosely follow the terminology of \cite{B}.
\begin{defi}
Let $R$ be a complex unimodular $2\times 6$ matrix, and assume that its rows are complex orthogonal. We call the pair of rows $\{r_1,r_2\}$ cancelling, if there exists column indices $p,q\in\{1,2,3,4,5,6\}$, such that $r_{1p}/r_{2p}+r_{1q}/r_{2q}=0$.
\end{defi}
This terminology is motivated by the following: if  $\{r_1,r_2\}$ is cancelling, then there exists a column permutation $\sigma=(\sigma_1,\sigma_2,\sigma_3,\sigma_4,\sigma_5,\sigma_6)$, such that: \[r_{1\sigma_1}/r_{2\sigma_1}+r_{1\sigma_2}/r_{2\sigma_2}=r_{1\sigma_3}/r_{2\sigma_3}+r_{1\sigma_4}/r_{2\sigma_4}=r_{1\sigma_5}/r_{2\sigma_5}+r_{1\sigma_6}/r_{2\sigma_6}=0,\]
that is, the six terms in the inner product $\left\langle r_1,r_2\right\rangle$ pairwise cancel each other out. Consequently, the cancelling property is invariant up to equivalence. Furthermore, the following polynomial captures whether the rows $\{h_x,h_y\}$ of $H$ are cancelling ($x,y\in\{1,2,3,4,5,6\}$ are distinct):
\begin{equation}\label{Cxy}
\mathcal{C}((x,y)):=\prod_{q=2}^6(h_{x1}h_{yq}+h_{xq}h_{y1}).
\end{equation}
Indeed, this is a simple consequence of Lemma~\ref{L22}. We have the following.
\begin{lemma}\label{LCxy}
Let $H=[h_{uv}]$ be a complex Hadamard matrix of order $6$. Let $i$, $j\in\{1,2,3,4,5,6\}$ be distinct. The rows $\{h_i,h_j\}$ are cancelling if and only if $\mathcal{C}((i,j))=\mathcal{C}((j,i))=0$.
\end{lemma}
\begin{proof}
Assume that $\mathcal{C}((i,j))=0$. Then, there exists $q\in\{2,3,4,5,6\}$, such that $h_{i1}h_{jq}+h_{iq}h_{j1}=0$. Since $|h_{uv}|=1$, this is equivalent to $h_{i1}/h_{j1}+h_{iq}/h_{jq}=0$. Thus, the pair of rows $\{h_i,h_j\}$ is cancelling.

Conversely, assume that the pair of rows $\{h_i,h_j\}$ is cancelling. Then, using Lemma~\ref{L22}, these rows may be partitioned (after suitable column permutations) to three disjoint $2\times 2$ complex Hadamard matrices. Thus, in particular, there exists a column index $q\neq 1$, such that the $2\times 2$ matrix $\left[\begin{smallmatrix} h_{i1} & h_{iq}\\ h_{j1} & h_{jq}\end{smallmatrix}\right]$ is complex Hadamard. Thus $h_{i1}/h_{j1}+h_{iq}/h_{jq}=0$, consequently formula \eqref{Cxy} evaluated at $(i,j)$ is $0$.
\end{proof}
\begin{example}
The Di\c{t}\u{a} family \cite{D} with $|c|=1$:
\begin{equation}\label{D6c}
D_6(c)=\left[
\begin{array}{rrrrrr}
 1 & 1 & 1 & 1 & 1 & 1 \\
 1 & -1 & i & -i & -i c & i c \\
 1 & i & -1 & i & -i & -i \\
 1 & -i & i & -1 & i c & -i c \\
 1 & -i/c & -i & i/c & -1 & i \\
 1 & i/c & -i & -i/c & i & -1
\end{array}
\right],\ i^2+1=0,
\end{equation}
is a subfamily of $X_6$ (see \cite{SZ1}). It is easy to see that the rows of the matrix $D_6(c)$ are pairwise cancelling.
\end{example}
Matrices whose rows are pairwise cancelling are called regular Hadamard matrices, in the sense of \cite{B}. It is natural to try to classify all such matrices.
\begin{lemma}[Cf.~\mbox{\cite[Lemma~9.1]{B}}]\label{L2}
Let $R$ be a $3\times 6$ complex unimodular matrix with pairwise cancelling rows. Then, $R$ is a member of exactly one of the following two families of matrices, up to equivalence:
\[R_1(q)=\left[
\arraycolsep=4pt\begin{array}{rrrrrr}
1 & 1 & 1 & 1 & 1 & 1 \\
1 & \mathbf{i} & 1 & -\mathbf{i} & -1 & -1 \\
1 & -1 & -\mathbf{i} & \mathbf{i} & q & -q \\
\end{array}
\right],\quad R_2(q)=\left[
\arraycolsep=4pt\begin{array}{rrrrrr}
	1 & 1 & 1 & 1 & 1 & 1 \\
	1 & \mathbf{i} & -1 & -\mathbf{i} & q & -q \\
	1 & -1 & -q & \mathbf{i}q & -\mathbf{i}q & q \\
\end{array}
\right],\]
where $|q|=1$, and $\mathbf{i}$ is the complex imaginary unit.
\end{lemma}
\begin{proof}
It is not too hard to see that the first three families listed in \cite{B} are equivalent to $R_1(q)$, while the fourth one is equivalent to $R_2(q)$, after changing the variable $q$ if it is necessary. Further, one may compute the Haagerup invariant set \cite{UH} to symbolically verify that these two families are indeed inequivalent. We spare the reader the details.
\end{proof}
One of the main results of \cite{B} is the complete characterization of $6\times 6$ regular complex Hadamard matrices.
\begin{theorem}[See \cite{B}]\label{TBanica}
A complex Hadamard matrix $H$ of order $6$ belongs to the Di\c{t}\u{a} family \eqref{D6c} if and only if it is regular, that is, every pair of rows of it is cancelling.
\end{theorem}
In what follows we strengthen this result, and prove that if $H$ contains only a triplet of pairwise cancelling rows, then $H$ necessarily belongs to the Di\c{t}\u{a} family.
\begin{proposition}\label{PC1}
Consider the matrix $R_1(q_0)$ for some fixed complex unimodular number $q_0$. If $s$ is a row with unimodular entries, which is pairwise orthogonal to all three rows of $R_1(q_0)$, then the pairs $\{r,s\}$ are cancelling for every row $r$ of $R_1(q_0)$.
\end{proposition}
\begin{proof}
Assume that the following matrix (see $R_1(q)$ in Lemma~\ref{L2}):
\[R=\left[
\begin{array}{rrrrrr}
1 & 1 & 1 & 1 & 1 & 1 \\
1 & i & 1 & -i & -1 & -1 \\
1 & -1 & -i & i & q & -q \\
1 & a & b & c & d & e
\end{array}
\right],\ i^2+1=0,\]
forms the first four rows of a complex Hadamard matrix $H$. We will show that if the fourth row $r_4$ is complex orthogonal to the first three rows $r_j$, $j\in\{1,2,3\}$ of $R$ then all the pairs $\{r_4,r_j\}$ are necessarily cancelling. 

In order to prove that the pair of rows $\{r_4,r_j\}$ is indeed cancelling, we consider the following system of $14$ complex polynomial equations (see Remark~\ref{newremark}) in the variables $a,b,c,d,e,i,q,u$ with rational coefficients:
\[\begin{cases}
\mathcal{O}((x,y))=0,\ (x,y)\in\{1,2,3,4\}^2,\ x\neq y\\
i^2+1=0,\\
uabcdeq\cdot \mathcal{C}((4,j))+1=0,
\end{cases}\]
describing the situation when the pair $\{r_j,r_4\}$ is not cancelling. The auxiliary variable $u$ ascertains that the variables as well as $\mathcal{C}((4,j))$ are nonzero. For each $j\in\{1,2,3\}$ we compute a reduced Gr\"obner basis within a few seconds, which turns out to be $\{1\}$. Hilbert's Nullstellensatz then implies that these system of equations have no solutions. Therefore if the rows of $R$ are complex orthogonal, then they are necessarily pairwise cancelling.
\end{proof}
Rather remarkably, the matrix $W$ in \eqref{W}, whose first three rows coincide with $R_2((1-\mathbf{i})/\sqrt{2})$, shows that Proposition~\ref{PC1} does not generalize to the family $R_2(q)$. Indeed, $\{w_1,w_4\}$ is not cancelling, which can be readily seen by evaluating $\mathcal{C}((1,4))\neq0$. However, once we restrict the statement to submatrices of $6\times 6$ Hadamard matrices (in particular subject to the additional vanishing conditions described in Lemma~\ref{L1}), then the cancelling property automatically follows from orthogonality.
\begin{proposition}\label{PC2}
Consider the matrix $R_2(q_0)$ for some fixed complex unimodular number $q_0$. If $s$ is a row with unimodular entries, which is pairwise orthogonal to all three rows of $R_2(q_0)$, and the rows and columns of the augmented matrix $(r_1,r_2,r_3,s)$ satisfy the vanishing conditions of Lemma~\ref{L1}, then the pair of rows $\{r_j,s\}$ is cancelling for each $j\in\{1,2,3\}$.
\end{proposition}
\begin{proof}
The proof is very similar to the proof of Proposition~\ref{PC1}. Assume that the following matrix (see $R_2(q)$ in Lemma~\ref{L2}):
\[R=\left[
\begin{array}{rrrrrr}
1 & 1 & 1 & 1 & 1 & 1 \\
1 & i & -1 & -i & q & -q \\
1 & -1 & -q & iq & -iq & q \\
1 & a & b & c & d & e
\end{array}
\right],\ i^2+1=0,\]
forms the first four rows of a complex Hadamard matrix $H$. Then we consider the system of orthogonality equations \eqref{ort} together with the Haagerup polynomial \eqref{POLYH} associated to the leading $4\times 3$ submatrix of $R$ in the variables $a,b,c,d,e,i,q,u$:
\[\begin{cases}
\mathcal{O}((x,y))=0,\ (x,y)\in\{1,2,3,4\}^2,\ x\neq y\\
\mathcal{H}((1,2,3),(1,2,3,4))=0,\ \text{for $R^T$}\\
i^2+1=0,\\
uabcdeq\cdot \mathcal{C}((4,j))+1=0.
\end{cases}\]
This system of $15$ polynomial equations describes the situation where the pair of rows $\{r_j,r_4\}$ is not cancelling (where $j\in\{1,2,3\}$ is fixed). Similarly as before, three $1$-minute Gr\"obner basis calculations show that these systems have no solutions. Therefore $\{r_j,r_4\}$ is necessarily cancelling for each $j\in\{1,2,3\}$.
\end{proof}
The main result of this section is the following.
\begin{theorem}\label{TTriangle}
A complex Hadamard matrix $H$ of order $6$ belongs to the Di\c{t}\u{a} family \eqref{D6c} if and only if it has a triplet of pairwise cancelling rows.
\end{theorem}
\begin{proof}
We may assume, using Lemma~\ref{L2}, that the first three rows of $H$ are either $R_1(q_0)$ or $R_2(q_0)$ for some fixed complex unimodular number $q_0$. Let us denote the rows of $H$ by $r_i$, $i\in\{1,2,3,4,5,6\}$. By Propositions \ref{PC1} and \ref{PC2}, $\{r_i,r_4\}$ is cancelling for $i\in\{1,2,3\}$, thus the first four rows are pairwise cancelling. Analogously, $\{r_i,r_5\}$ is cancelling for $i\in\{1,2,3\}$, and $i\in\{1,2,4\}$, hence the first five rows are pairwise cancelling. Analogously, all six rows are pairwise cancelling. Therefore, referring to Theorem~\ref{TBanica}, $H$ belongs to the Di\c{t}\u{a} family.
\end{proof}
\section{The case of three rows: the 2-circulant family}\label{S2COLS}
In this section we consider $6\times 6$ complex Hadamard matrices having a submatrix:
\begin{equation}\label{diagonalpattern}\left[\begin{array}{rrrr}1 &1&1&1\\
1 & -1 &\ast &\ast\\
1 & \ast & -1 & \ast\\
1 & \ast & \ast &-1\end{array}\right].
\end{equation}
That is, the matrix has three distinct columns, each containing a $-1$ entry, which are furthermore located in three distinct rows. There have been notable previous efforts studying matrices with similar properties \cite{BN}, \cite{MSZ}, \cite{LEE}.
\begin{proposition}\label{P51}
Let $H$ be a complex Hadamard matrix of order $6$, such that its upper-left $4\times 4$ submatrix has the following form:
\[M=\left[\begin{array}{rrrr}
1 & 1 & 1 & 1\\
1 & -1 & a & b\\
1 & c & -1 & d\\
1 & e & f & -1\end{array}\right],\]
where $(a+b)(c+d)(e+f)(c+e)(a+f)(b+d)=0$. Then $H$ belongs to the Di\c{t}\u{a} family \eqref{D6c}.
\end{proposition}
\begin{proof}
We may assume, by applying suitable permutations, relabelling, and transposing $H$ if necessary, that $d=-b$. Consequently the first three rows of $H$ are pairwise cancelling. Therefore, by Theorem~\ref{TTriangle}, the matrix must belong to the Di\c{t}\u{a} family. As the Di\c{t}\u{a} family is closed with respect to transposition, indeed, $D_6^T(c)=D_6(1/c)$, we are done.
\end{proof}
In order to move forward, first we need the following technical result.
\begin{lemma}\label{L51}
Let $H$ be a complex Hadamard matrix of order $6$, such that its upper-left $4\times 4$ submatrix has the following form:
\[M=\left[\begin{array}{rrrr}
1 & 1 & 1 & 1\\
1 & -1 & a & b\\
1 & c & -1 & d\\
1 & e & f & -1\end{array}\right],\]
where $(a+b)(c+d)(e+f)(c+e)(a+f)(b+d)\neq 0$. Then $ac=be=df$.
\end{lemma}
\begin{proof}
The proof is computational, using the polynomial equations \eqref{POLYH} for $M$ and its transpose. We will prove that $ac-be\neq 0$ is impossible. Therefore $ac-be=0$, and by symmetry, $ac-df=0$, and $be-df=0$. Consider the following system of $17$ polynomial equations in the complex variables $a,b,c,d,e,f,u$, where
\[\begin{cases}
\mathcal{H}((i,j,k),(1,2,3,4))=0, (i,j,k)\in\{1,2,3,4\}^3, i<j<k \text{ or } i>j>k\\
\mathcal{H}((i,j,k),(1,2,3,4))=0, (i,j,k)\in\{1,2,3,4\}^3, i<j<k \text{ or } i>j>k,\ \text{for }M^T\\
uabcdef(a+b)(c+d)(e+f)(c+e)(a+f)(b+d)(ac-be)+1=0.
\end{cases}\]
Similarly as before, the auxiliary variable $u$ ensures that a certain polynomial expression is nonzero. A $9.5$-hours long Gr\"obner basis computation in Mathematica (with peak memory usage about 2.6GB) testifies that this system of equations has no solutions. Therefore, since $abcdef(a+b)(c+d)(e+f)(c+e)(a+f)(b+d)\neq 0$, we necessarily have $ac=be$.
\end{proof}
\begin{remark}
We would be interested to see a result similar in spirit to Lemma~\ref{L51} concerning the more general situation of $4\times 4$ submatrices with no vanishing row or column sums having generic diagonal elements $x$, $y$, and $z$.
\end{remark}
The following is the main result of this section.
\begin{proposition}\label{P3}
Let $H$ be a complex Hadamard matrix of order $6$, such that its upper-left $4\times 4$ submatrix has the following form:
\[M=\left[\begin{array}{rrrr}
1 & 1 & 1 & 1\\
1 & -1 & a & b\\
1 & c & -1 & d\\
1 & e & f & -1\end{array}\right],\]
where $(a+b)(c+d)(e+f)(c+e)(a+f)(b+d)\neq 0$. Then $H$ belongs to the $2$-circulant family \eqref{X6alpha}.
\end{proposition}
In the proof we perform repeatedly routine Gr\"obner basis calculations (with respect to a conveniently chosen term order) aiming to discover low-degree algebraic relations between the matrix entries. The interested reader is advised to use a suitable computer algebra system of their choice for bookkeeping purposes.
\begin{proof}
Since $M$ does not contain any vanishing rows, we may complete the first four rows of $H$ formally in exactly $8$ ways by putting $\left[\begin{array}{cc}1&1\end{array}\right]$ in the first row followed by any of the patterns $\left[\begin{smallmatrix}-a &-b\\ -c & -d\\ -e & -f\end{smallmatrix}\right], \dots, \left[\begin{smallmatrix}-b & -a\\ -d & -c\\ -f & -e\end{smallmatrix}\right]$ in subsequent rows. We are free to choose the order of the rows and columns of $H$ up to permutations. Using Lemma~\ref{L51}, which says that $\Delta:=ac=be=df$, we readily see that irrespective of the pattern we choose, there exist (at least) two distinguished rows, in which the product of an entry in column $5$ with the entry in column $6$ of the other row equals $\Delta$. Furthermore, some of these distinguished rows may be permuted to the top so that $h_{25}=-h_{23}$ and $h_{36}=-h_{32}$, after appropriate column permutations maintaining the overall structure of $M$. Therefore we may assume that $H$ admits the following canonical form, up to equivalence (and relabelling if it is necessary):
	\[H^{(\xi,\zeta)}(a,b,c,d,w,x,y,z)=\left[\begin{array}{cccrrr}
		1 & 1 & 1 & 1 & 1 & 1\\
		1 & -1 & a & b & -a & -b\\
		1 & c & -1 & d & -d & -c\\
		1 & ac/b & ac/d & -1 & \multicolumn{2}{c}{U^{(\xi)}}\\
		1 & -c & \multicolumn{2}{c}{\multirow{2}{*}{$L^{(\zeta)}$}}          & w & x\\
		1 & -ac/b    &    & & y & z\end{array}\right].\]
Here for $\xi\in\{1,2\}$, $U^{(1)}=\left[\begin{array}{cc} -ac/b & -ac/d\end{array}\right]$ and $U^{(2)}=\left[\begin{array}{cc} -ac/d & -ac/b\end{array}\right]$. Similarly, for $\zeta\in\{1,2,3,4\}$, the missing lower-left $2\times 2$ submatrix $L^{(\zeta)}$ means the $\zeta$th matrix from the left in the following sequence:
\[\left[\begin{array}{cc}
		-a & -b\\
		-ac/d & -d\\
	\end{array}\right], \left[\begin{array}{cc}
		-a & -d\\
		-ac/d & -b\\
	\end{array}\right],\left[\begin{array}{cc}
		-ac/d & -b\\
		-a & -d\\
	\end{array}\right], \left[\begin{array}{cc}
		-ac/d & -d\\
		-a & -b\\
	\end{array}\right].\]
We will analyse each of these $8$ patterns by looking at the system of polynomial equations in the matrix entries $a,b,c,d,ac/b,ac/d,w,x,y,z,u$ formed by the orthogonality equations \eqref{ort} (also applied to $H^T$):
\[\begin{cases}
\mathcal{O}((i,j))=0,\ (i,j)\in\{1,2,3,4,5,6\}^2,\ i\neq j\\
\mathcal{O}((i,j))=0,\ (i,j)\in\{1,2,3,4,5,6\}^2,\ i\neq j,\ \text{for $H^T$}\\
uabcdwxyz(a+b)(c+d)(b+d)(a^2-4a+1)+1=0.
\end{cases}
\]
Here $u$ is an auxiliary variable, and we are preemptively discarding those phantom solutions where $a^2-4a+1=0$. By multiplying these equations by appropriate powers of $b$ and $d$, we obtain $4\cdot 15+1=61$ polynomial equations in the complex variables $a,b,c,d,w,x,y,z,u$. Since this system is rather complicated, we analyse its solutions in two phases.

In the first phase we consider the subsystem of those $24$ polynomials which do not feature any of the variables $u,w,x,y,z$, and, in addition, we include a new polynomial equation $uabcd(a+b)(c+d)(b+d)(a^2-4a+1)+1=0$. We compute a Gr\"obner basis with the goal of discovering several low-degree relations between the matrix entries $a$, $b$, $c$, and $d$. 

Using these relations (if any), in the second phase we substitute back into the full system so that to reduce its complexity, and we compute a Gr\"obner basis once again to find additional relations between the remaining variables. This process leads to the complete solution of the system of equations above.

It turns out that the analysis of the $8$ patterns $H^{(\xi,\zeta)}$ may be grouped into five cases, where the results are analogous. We describe these as follows.
	
{\bfseries Case 1}: The patterns $H^{(2,1)}$ and $H^{(2,2)}$ do not lead to any complex unimodular solutions, because the reduced Gr\"obner basis after the first phase is $\{1\}$. 
	
{\bfseries Case 2}:	We claim that every complex Hadamard matrix with pattern $H^{(1,3)}$ belongs to the $2$-circulant family $X_6$, see \eqref{X6alpha}. We found no useful relation during the first phase. In the second phase, a degree reverse lexicographic reduced Gr\"obner basis can be computed in about $3$ minutes giving $z=w=xy=ac$. After performing these substitutions in $H^{(1,3)}$ so that to eliminate $z$, $w$, and $y$, we find that the orthogonality equation $\mathcal{O}((2,5))=0$ implies that $x=-bc/d$. After plugging in the variable $x$ as indicated, there remains a pair of polynomial equations completely describing orthogonality of the matrix $H^{(1,3)}(a,b,c,d)$, namely:
\begin{equation}\label{H33def}
ac+bc-d+bd+cd-acd=0,
\end{equation}
and what is essentially its complex conjugate under the assumption that $|a|=|b|=|c|=|d|=1$: $b-ab-ac+abc-ad-bd=0$. Given a quadruple of complex unimodular numbers $(a,b,c,d)$, subject to \eqref{H33def}, we define the unimodular numbers
\[\beta^\ast:=-\left(\frac{d}{abc^2}\right)^{1/3},\quad \gamma^\ast:=-\frac{bc}{d}(\beta^\ast)^2,\quad \varepsilon^\ast:=-c\beta^\ast,\quad \varphi^\ast:=\frac{ac^2}{d}(\beta^\ast)^2,\]
where $\beta^\ast$ is any fixed complex number such that $(\beta^\ast)^3=-d/(abc^2)$. Using \eqref{H33def}, we readily see that the quadruple $(\beta^\ast, \gamma^\ast, \varepsilon^\ast,\varphi^\ast)$ satisfies \eqref{implicit}, and consequently plugging into \eqref{X6alpha} we see that $X_6(\beta^\ast,\gamma^\ast,\varepsilon^\ast,\varphi^\ast)=H^{(1,3)}(a,b,c,d)$ is equivalent to a $2$-circulant complex Hadamard matrix.
	
{\bfseries Case 3}:	We claim that the pattern $H^{(1,2)}$ leads to a subfamily of $X_6$. In the first phase we find that $b=a$, and subsequently $x=y=d^2$, $z=w=-d$. Then, from $\mathcal{O}((2,5))=0$ we find that $c=d^2/a$. We are left with a single polynomial (whose complex conjugate is essentially itself)
\begin{equation}\label{subfam}
a^2-a+2ad+d^2-ad^2=0
\end{equation}
ensuring orthogonality.	By permuting the last two rows of the resulting matrix $H^{(1,2)}(a,d)$ gives exactly $H^{(1,3)}(a,a,d^2/a,d)$ from Case 2 above, and plugging in $b\leftarrow a$ and $c\leftarrow d^2/a$ to equation \eqref{H33def} yields \eqref{subfam}, up to a negligible constant factor.
	
{\bfseries Case 4}: We move on to discuss the case $H^{(1,1)}$. In the first phase we find that $d=c$. After performing this substitution the second phase yields $x=y=-b$, and $z=w=b^2$. From $\mathcal{O}((2,5))=0$ we have $c=b^2/a$, and there remains a single relation between the variables $a$ and $b$:
\begin{equation}\label{sub2}
a-a^2-2ab-b^2+ab^2=0.
\end{equation}
Now $H^{(1,1)}(a,b)=H^{(1,3)}(a,b,b^2/a,b^2/a)$ and plugging in $c\leftarrow b^2/a$, $d\leftarrow b^2/a$ to equation \eqref{H33def} yields \eqref{sub2}, up to a constant factor. Thus, this is a subfamily of $X_6$.

{\bfseries Case 5}: In each of the cases $(\xi,\zeta)\in\{(1,4),(2,3),(2,4)\}$ we discover after the first phase $d=b$ and $c=a$. Using this, we get $y=x$, $z=w$ and then $x=-a$ and $w=a^2$. After these variable substitutions all three matrix patterns become identical, and the remaining two parameters $a$ and $b$ are related via the single equation $a^2-b+2ab-a^2b+b^2=0$. Interchanging the variables $a\leftrightarrow b$, the rows $r_3\leftrightarrow r_4$ and $r_5\leftrightarrow r_6$, and the columns $c_3\leftrightarrow c_4$, and $c_5\leftrightarrow c_6$ yields a matrix identical to what was described earlier in Case 4. We have a subfamily of $X_6$.
\end{proof}
We summarize the results of this section in the following.
\begin{theorem}\label{T54}
Let $H$ be a complex Hadamard matrix of order $6$. Then the following are equivalent:
\begin{enumerate}
\item[($a$)] every normalized matrix in the equivalence class of $H$ contains a (not necessarily contiguous) submatrix of the form \eqref{diagonalpattern};
\item[($b$)] some normalized matrix in the equivalence class of $H$ contains a (not necessarily contiguous) submatrix of the form \eqref{diagonalpattern};
\item[($c$)] $H$ is a member of the $2$-circulant family \eqref{X6alpha}, up to equivalence.
\end{enumerate}
\end{theorem}
\begin{proof}
The direction $(a)\Rightarrow(b)$ is trivial, and $(b)\Rightarrow(c)$ follows from the combination of Propositions~\ref{P51} and \ref{P3}, by noting that the Di\c{t}\u{a} family is a subfamily of the $2$-circulant family. Therefore, it remains to show that $(c)\Rightarrow (a)$.

Let $H$ be a member of the $2$-circulant family $X_6$, up to equivalence. By using the equivalence operations, we may assume that $H$ admits the following $2$-circulant structure (see \cite{SZ1}):
\begin{equation}\label{2CIRCFORM}
H=\left[\begin{array}{cc}
A&B\\
B^\ast& -A^\ast\end{array}\right],
\end{equation}
where $A$ and $B$ are $3\times 3$ circulant matrices, and $\ast$ denotes the conjugate transpose. 

Next observe that $H$ can be transformed into any equivalent, normalized matrix $L$ by the following three-step process: (i) permute the desired pivoting row and column of $H$ to the top and to the left, so that to become eventually the first row and column of $L$, respectively; (ii) normalize the obtained matrix; (iii) permute the noninitial rows and columns amongst themselves. We show that the resulting matrix $L$ has the required substructure \eqref{diagonalpattern}.

We may assume that the pivoting row and column intersect in block $A$ of \eqref{2CIRCFORM}, by interchanging the blocks and multiplying the last three rows by $-1$ if it is necessary. Note that this block exchange maintains the $2$-circulant pattern. Next consider the circulant permutation matrix $U=\left[\begin{smallmatrix}
0 & 1 & 0\\
0 & 0 & 1\\
1 & 0 & 0\end{smallmatrix}\right]$. Clearly, multiplying any $3\times 3$ matrix by $U$ on the right will cyclically shift its columns to the right; similarly, multiplying a matrix by $U$ on the left will cyclically shift its rows upwards. Therefore, multiplying $H$ by the block diagonal matrix $D=[U,U^\ast]$ on the right (or left) will cyclically shift the columns (or rows, respectively) of the blocks $A$ and $B$, while maintaining the overall $2$-circulant form. Repeated multiplication by $D$ will perform the desired step (i). We call the resulting intermediate matrix $K$.

In the matrix $K$ consider the $2\times 2$ submatrices  $\left[\begin{smallmatrix} k_{11} & k_{1j}\\
k_{j1} & k_{jj}\end{smallmatrix}\right]=\left[\begin{smallmatrix} k_{11} & k_{1j}\\
1/k_{1j} & -1/k_{11}\end{smallmatrix}\right]$, where $j\in\{4,5,6\}$. Now let's perform step (ii), that is, normalize $K$ by first multiplying its first row and column by arbitrary complex unimodular numbers $x$ and $y$, respectively, such that $xyk_{11}=1$. We obtain the following matrices:
\[\left[\begin{array}{cc}
k_{11} & k_{1j}\\
1/k_{1j} & -1/k_{11}\\
\end{array}\right]\leadsto\left[\begin{array}{cc}
xyk_{11} & xk_{1j}\\
y/k_{1j} & -1/k_{11}\\
\end{array}\right]\leadsto\left[\begin{array}{cc}
1 & 1\\
1 & -1/(xyk_{11})\\
\end{array}\right]=\left[\begin{array}{cr}
1 & 1\\
1 & -1\\
\end{array}\right].
\]
In order to complete step (ii), let's multiply the noninitial rows and columns by $1/(xk_{1j})$ and $k_{1j}/y$ respectively. Thus, for $j\in\{4,5,6\}$ we find three $-1$ entries in the normalized matrix just as in \eqref{diagonalpattern}. We arrive at the matrix $L$ by performing step (iii). This simply reshuffles the pattern \eqref{diagonalpattern}.
\end{proof}
\section{The case of two rows}\label{S3COLS}
In this section we consider $6\times 6$ complex Hadamard matrices having a submatrix:
\begin{equation}\left[\begin{array}{rrrr}\label{PART3}
1 &1&1&1\\
1 & -1 & -1 &\ast\\
1 & \ast & \ast & -1\end{array}\right].\end{equation}
That is, the matrix has three distinct columns, each containing a $-1$ entry, which are furthermore located in two distinct rows. We show that this pattern eventually leads to one of the previously described cases. The main tool we use is Karlsson's celebrated result \cite{K1} on $H_2$-reducible complex Hadamard matrices of order $6$.
\begin{theorem}[See \cite{K1}]\label{KARLSSON}
Let $H$ be a complex Hadamard matrix of order $6$. Assume that $H$ has a $2\times 2$ submatrix $K$ which is a complex Hadamard matrix. Then the rows and columns of $H$ can be permuted in a way so that $K$ becomes the top-left $2\times 2$ block, and all nine $2\times 2$ blocks become complex Hadamard matrices.
\end{theorem}
We proceed by considering normalized Hadamard matrices having a row containing at least two $-1$ entries. Since such matrices are $H_2$-reducible by Theorem~\ref{KARLSSON}, we may assume that they have the following form. Hereafter the horizontal and vertical lines emphasize the $H_2$-reducible partition of the matrices.
\begin{lemma}\label{L62}
Let $H$ be a complex Hadamard matrix of order $6$, such that its upper-left $4\times 4$ submatrix has the following form:
\begin{equation}\label{M}
M=\left[\begin{array}{cr|cc}
	1 & 1 & 1 & 1\\
	1 & -1 & 1 & -1\\
	\hline
	1 & b & p & q\\
	1 & -b & w & -qw/p
\end{array}\right].\end{equation}
Then necessarily $(b-qw)(1+w+w^2)=0$.
\end{lemma}
\begin{proof}
Consider the system of polynomial equations in the variables $b,p,q,w,u$:
\[\begin{cases}
\mathcal{P}_1:=\mathcal{H}((1,2,3),(1,2,3,4))&=0,\\
\mathcal{P}_2:=p^2\mathcal{H}((1,2,4),(1,2,3,4))&=0,\\
\mathcal{P}_3:=\mathcal{H}((1,2,3),(1,2,3,4))&=0,\quad \text{for $M^T$}\\
\mathcal{P}_4:=ubpqw+1&=0.
\end{cases}\]
If $M$ is a submatrix of a complex Hadamard matrix, then by Lemma~\ref{L1} the polynomials $\mathcal{P}_i$, $i\in\{1,2,3\}$ are simultaneously zero, and $\mathcal{P}_4$ has a root at $u=-1/(bpqw)$. Consider the four witness polynomials, discovered as a by-product of a Gr\"obner basis computation: $\mathcal{Q}_1:=-buw+quw^4$, $\mathcal{Q}_2:=-pquw-b^2qu^2w$,
\begin{align*}
\mathcal{Q}_3&:=-2quw+pquw-quw^2-pq^3u^2w^2+q^3u^2w^3-2pq^3u^2w^3,\\
\begin{split}
\mathcal{Q}_4&:=8b+8bw-8qw+2b^3puw+2b^2quw+2bpq^2uw+8bw^2-8qw^2\\
&\phantom{:=\ }+4b^2quw^2+8bpq^2uw^2+2bp^2q^2uw^2-8qw^3+2b^2quw^3\\
&\phantom{:=\ }+8bpq^2uw^3+4bp^2q^2uw^3-2bq^2uw^4+4bpq^2uw^4.
\end{split}
\end{align*}
One can verify that $\sum_{i=1}^4\mathcal{P}_i\mathcal{Q}_i=8(b-qw)(1+w+w^2)$, a routine task using computer algebra. Therefore, if the polynomials $\mathcal{P}_i$ are simultaneously $0$, then the right-hand side also must be $0$.
\end{proof}
\begin{remark}
Let $H$ be a complex Hadamard matrix of order $6$, such that its upper-left $4\times 4$ submatrix has the following generic form:
\[\left[\begin{array}{cr|cc}
	1 & 1 & 1 & 1\\
	1 & -1 & a & -a\\
	\hline
	1 & b & p & q\\
	1 & -b & w & -qw/p
\end{array}\right].\]
Then, one may discover that $(1+a)(ab-qw)(1+w+w^2)(a^2+aw+w^2)=0$. While the proof we aware of is similar to that of Lemma~\ref{L62}, the witness polynomials are considerably more complicated. We would be very interested in understanding the reasons behind this identity. The special case $a=1$ leads to the more accessible, and in particular, human-verifiable Lemma~\ref{L62}.
\end{remark}
Now we turn to the classification of Hadamard matrices with the substructure \eqref{M}. There are two cases in the setting of Lemma \ref{L62}, either $1+w+w^2=0$ or $b=qw$. First we consider the former case.
\begin{proposition}\label{PROP63}
Assume that $H$ is a $6\times 6$ complex Hadamard matrix with submatrix $M$ as in \eqref{M}. Assume further that $b-qw\neq 0$. Then $H$ necessarily belongs to the transposed Fourier family, or to the following subfamily of the Fourier matrices:
\begin{equation}\label{Ha}
H(a)=\def\arraystretch{1.2}\left[\begin{array}{cr|cr|cr}
1 & 1 & 1 & 1 & 1 & 1\\
1 & -1 & 1 & -1 & a & -a\\
\hline
1 & 1 & w & w & w^2 & w^2\\
1 & -1 & w & -w & aw^2 & -aw^2\\
\hline
1 & 1 & w^2 & w^2 & w & w\\
1 & -1 & w^2 & -w^2 & aw & -aw\\ 
\end{array}\right],\ 1+w+w^2=0,\ a^2\neq 1.
\end{equation}
\end{proposition}
\begin{proof}
First we exploit the symmetries of the matrix \eqref{M}. Three additional matrices can be obtained with similar structure as follows: (a) swapping the last two rows; (b) swapping the first two columns and the last two columns followed by normalization; (c) performing (a) and then (b). Applying Lemma~\ref{L62} to these four matrices, subject to $b-qw\neq 0$ leads to the system of polynomial equations:
\[1+w+w^2=1+p+p^2=(bp)^2+bpqw+(qw)^2=b^2+bq+q^2=u(b-qw)+1=0.\]
Therefore: $0=-bqu(1+w+w^2)+uw(b^2+bq+q^2)+(q-bw)\big[u(b-qw)+1\big]=q-bw$. If $p=w^2$, then $(bp)^2+bpqw+(qw)^2=3b^2w^4\neq 0$. Therefore $p=w$, and we can extend the matrix $M$ with further two columns, to obtain the $H_2$-reducible pattern
\[\left[\begin{array}{cr|cr}
	1 & 1 & 1 & 1\\
	1 & -1 & 1 & -1\\
	\hline
	1 & b & w & bw\\
	1 & -b & w & -bw\\
\end{array}\right]\leadsto\left[\begin{array}{cr|cr|cc}
1 & 1 & 1 & 1 & 1 & 1\\
1 & -1 & 1 & -1 & a & -a\\
\hline
1 & b & w & bw & x & y\\
1 & -b & w & -bw & z & -yz/x\\
\end{array}\right],\]
where all $2\times 2$ blocks are complex Hadamard. Here, if $a^2=1$, then by Theorem~\ref{F6CHAR} the matrix belongs to the transposed Fourier family, and we are done. Otherwise, we may assume that $a^2\neq 1$ and proceed to consider the orthogonality equations (after appropriate scaling by $x$ if necessary):
\begin{equation}\label{prevstar}\begin{cases}
\mathcal{O}((i,j))=0,\ (i,j)\in\{1,2,3,4\}^2,\ i\neq j\\
1+w+w^2=0,\\
uabxyz(a^2-1)+1=0.
\end{cases}\end{equation}
We compute a Gr\"obner basis and discover $b^2=1$. By permuting rows $3$ and $4$ (and changing the variables $x$, $y$, $z$), we may assume that $b=1$ and consequently $x=y=w^2$. Next, we fill out the last two rows, and obtain
\[\def\arraystretch{1.2}\left[\begin{array}{cr|cr|cc}
1 & 1 & 1 & 1 & 1 & 1\\
1 & -1 & 1 & -1 & a & -a\\
\hline
1 & 1 & w & w & w^2 & w^2\\
1 & -1 & w & -w & z & -z\\
\hline
1 & c & w^2 & d\\
1 & -c & w^2 & -d\\
\end{array}\right]\leadsto\left[\begin{array}{cr|cr|cc}
1 & 1 & 1 & 1 & 1 & 1\\
1 & -1 & 1 & -1 & a & -a\\
\hline
1 & 1 & w & w & w^2 & w^2\\
1 & -1 & w & -w & aw^2 & -aw^2\\
\hline
1 & c & w^2 & cw^2 & s & t\\
1 & -c & w^2 & -cw^2 & v & -tv/s\\ 
\end{array}\right].\]
By considering $\mathcal{O}((2,4))=0$, we have $z=aw^2$. By the orthogonality of columns $2$ and $4$, we have $d=cw^2$. Finally, we fill out the lower right submatrix, and observe that the set of rows $\{r_1,r_2,r_5,r_6\}$ forms the exact same instance as the structure previously investigated in \eqref{prevstar}. Hence $c^2=1$ is the only possibility. By permuting the last two rows (and relabelling the remaining variables), we have $c=1$, $s=t=w$, $v=aw$, and $a$ remains a free parameter as in the statement.
\end{proof}
Finally, the last remaining case is when $b-qw=0$ in the pattern \eqref{M}. From now on we will also assume that an additional row contains a $-1$ entry as in \eqref{PART3}.
\begin{theorem}\label{T64}
Assume that $H$ is a $6\times 6$ complex Hadamard matrix having a submatrix \eqref{PART3}, up to equivalence. Then $H$ belongs to the transposed Fourier family \eqref{Fab}, or to the $2$-circulant family \eqref{X6alpha}.
\end{theorem}
\begin{proof}
We may permute the rows and columns of $H$ so that its upper left $4\times 4$ submatrix is of the form \eqref{M}, where, in addition, the third or fourth row contains a $-1$ entry at positions $3$, $5$ or $6$. In particular, Lemma~\ref{L62} applies.

If, in addition, $b-qw\neq0$, then Proposition~\ref{PROP63} is in effect, and either we are dealing with the transposed Fourier family, up to equivalence; or with the subfamily $H(a)$ exactly as given in \eqref{Ha}. In this family of matrices additional $-1$ entries may appear when the parameter $a\in\{w,-w,w^2,-w^2\}$. However, in all of these cases $H(a)$ will contain three $-1$ entries aligned as in pattern \eqref{diagonalpattern}, hence by Theorem~\ref{T54} these matrices belong to the $2$-circulant family.

Therefore, we may consider the last remaining case $w=b/q$, and the matrices:
\[\left[\begin{array}{cr|cc|cc}
	1 & 1 & 1 & 1 & 1 & 1\\
	1 & -1 & 1 & -1 & a & -a\\
	\hline
	1 & b & p & q & x & y\\
	1 & -b & b/q & -b/p & z & -yz/x
\end{array}\right]\leadsto\left[\begin{array}{cr|cc|rr}
1 & 1 & 1 & 1 & 1 & 1\\
1 & -1 & 1 & -1 & a & -a\\
\hline
1 & b & p & q & -1 & y\\
1 & -b & b/q & -b/p & z & yz
\end{array}\right].\]
Now if $a^2=1$, then we must have the transposed Fourier family by Theorem~\ref{F6CHAR}. Therefore we may assume that $a^2\neq 1$. By swapping the last two rows and relabelling, we may assume that the additional $-1$ entry is in the third row. If $p=-1$, then the first three rows are clearly pairwise cancelling, and we must have the Di\c{t}\u{a} family by Theorem~\ref{TTriangle}. Therefore we proceed with the case $p\neq -1$, and up to permuting the last two columns and relabelling if necessary, we set $h_{35}=-1$. Since $\left\langle h_1,h_3\right\rangle=0$, Lemma~\ref{L22} gives $(y+b)(y+p)(y+q)=0$, leading to the following three cases.

{\bfseries Case 1}: First consider $y=-b$, $q=-p$, leading to the matrix:
\[\left[\begin{array}{cr|cc|rc}
	1 & 1 & 1 & 1 & 1 & 1\\
	1 & -1 & 1 & -1 & a & -a\\
	\hline
	1 & b & p & -p & -1 & -b\\
	1 & -b & -b/p & -b/p & z & -bz
\end{array}\right]\leadsto\left[
\begin{array}{cr|cc|rr}
1 & 1 & 1 & 1 & 1 & 1 \\
1 & -1 & 1 & -1 & a & -a \\
\hline
1 & b & p & -p & -1 & -b \\
1 & -b & -b/p & -b/p & -a & ab \\
\end{array}
\right].\]
We show that $z=-a$. If $b\neq -1$, then this follows immediately from $\left\langle h_2,h_4\right\rangle=0$. Otherwise if $b=-1$, then $\left\langle h_1,h_4\right\rangle=0$ implies $1+z+pz=0$, thus $pz=z^2$, $p=z$. Finally, $0=\left\langle h_1,h_4\right\rangle-\left\langle h_2,h_3\right\rangle=2z+2a$ leads to $z=-a$ as claimed.

Next, we have $0=abp\left\langle h_1,h_4\right\rangle+b p\left\langle h_2,h_3\right\rangle=2b-2ap^2$, and we get the matrix:
\[\left[
\arraycolsep=3pt\def\arraystretch{1.2}\begin{array}{cc|cc|rc}
	1 & 1 & 1 & 1 & 1 & 1 \\
	1 & -1 & 1 & -1 & a & -a \\
	\hline
	1 & ap^2 & p & -p & -1 & -ap^2 \\
	1 & -ap^2 & -ap & -ap & -a & a^2p^2 \\
\end{array}
\right]\leadsto \left[
\def\arraystretch{1.2}\begin{array}{cc|cc|rc}
	1 & 1 & 1 & 1 & 1 & 1 \\
	1 & -1 & 1 & -1 & a & -a \\
	\hline
	1 & ap^2 & p & -p & -1 & -ap^2 \\
	1 & -ap^2 & -ap & -ap & -a & a^2p^2 \\
	\hline
	1 & c  & f & g &  & \\
	1 & -c & h & -gh/f &  & \\
\end{array}
\right].\]
We extend its first four columns, and consider the column orthogonality equations (after scaling by $f$ if it is necessary):
\[\begin{cases}
\mathcal{O}((i,j))=0,\ (i,j)\in\{1,2,3,4\}^2,\ i\neq j\ \text{ (for $H^T$)}\\
uacfghp(a^2-1)+1=0.
\end{cases}\]
We compute a Gr\"obner basis to discover $c^2=1$. Since both $c$ and $-c$ are in this matrix, it is equivalent to one with three $-1$s on the diagonal. Therefore, by Theorem~\ref{T54}, $H$ belongs to the $2$-circulant family.

{\bfseries Case 2}: If $y=-p$, $q=-b$, then we have:
\[\left[\begin{array}{cr|rc|rc}
	1 & 1 & 1 & 1 & 1 & 1\\
	1 & -1 & 1 & -1 & a & -a\\
	\hline
	1 & b & p & -b & -1 & -p\\
	1 & -b & -1 & -b/p & z & -pz
\end{array}\right],\]
and we immediately see that there are three distinct rows and columns containing a $-1$ entry. Therefore, by Theorem~\ref{T54}, $H$ belongs to the $2$-circulant family.

{\bfseries Case 3}: Finally, let $y=-q$, $p=-b$. This leads to the matrix:
\[\left[\begin{array}{cr|cr|rc}
	1 & 1 & 1 & 1 & 1 & 1\\
	1 & -1 & 1 & -1 & a & -a\\
	\hline
	1 & b & -b & q & -1 & -q\\
	1 & -b & b/q & 1 & z & -qz
\end{array}\right]\leadsto\left[
\begin{array}{cr|rr|rr}
 1 & 1 & 1 & 1 & 1 & 1 \\
 1 & -1 & 1 & -1 & a & -a \\
\hline
 1 & b & -b & -1 & -1 & 1 \\
 1 & -b & -b & 1 & z & z \\
\end{array}
\right].\]
First, observe that column $4$ must contain at least two $-1$ entries. If $q\neq -1$, then there must be a $-1$ in rows $5$ or $6$ and therefore by Theorem~\ref{T54}, $H$ belongs to the $2$-circulant family. 

Otherwise, if $q=-1$, then $0=(1-a)bz\left\langle h_1,h_4\right\rangle-b(1+z)\left\langle h_2,h_3\right\rangle=2(1+a z)$. Therefore $z=-1/a$. Further, $\left\langle h_1,h_4\right\rangle/2=1+(-a)+(-1/b)=0$, and hence $a^2=-1/b$, equivalently $b=-1/a^2$.
\[\left[
\begin{array}{cr|rr|cc}
 1 & 1 & 1 & 1 & 1 & 1 \\
 1 & -1 & 1 & -1 & a & -a \\
\hline
 1 & b & -b & -1 & -1 & 1 \\
 1 & -b & -b & 1 & -1/a & -1/a \\
\end{array}
\right]\leadsto\left[
\arraycolsep=4pt\def\arraystretch{1.2}\begin{array}{cc|cr|cc}
 1 & 1 & 1 & 1 & 1 & 1 \\
 1 & -1 & 1 & -1 & a & -a \\
\hline
 1 & -1/a^2 & 1/a^2 & -1 & -1 & 1 \\
 1 & 1/a^2 & 1/a^2 & 1 & -1/a & -1/a \\
\end{array}
\right].\]
Since the rows are orthogonal, we have $1-a+a^2=0$. Therefore the last two entries in the third column are both $-1/a$. Similarly, since $1-1/a-a=0$, the last two entries in the sixth column are $-a$ and $-1/a$ in some order. Therefore up to swapping the last two rows, we have:
\[\left[
\arraycolsep=3pt\def\arraystretch{1.2}\begin{array}{cr|cr|cc}
 1 & 1 & 1 & 1 & 1 & 1 \\
 1 & -1 & 1 & -1 & a & -a \\
\hline
 1 & -1/a^2 & 1/a^2 & -1 & -1 & 1 \\
 1 & 1/a^2 & 1/a^2 & 1 & -1/a & -1/a \\
\hline
	1 & c & -1/a & g & s & -a\\
	1 & -c & -1/a & -g & -s/a^2 & -1/a\\ 
\end{array}
\right]\leadsto\left[
\arraycolsep=3pt\def\arraystretch{1.2}\begin{array}{cr|cr|cc}
 1 & 1 & 1 & 1 & 1 & 1 \\
 1 & -1 & 1 & -1 & a & -a \\
\hline
 1 & -1/a^2 & 1/a^2 & -1 & -1 & 1 \\
 1 & 1/a^2 & 1/a^2 & 1 & -1/a & -1/a \\
\hline
	1 & c & -1/a & 1 & s & -a\\
	1 & -c & -1/a & -1 & -s/a^2 & -1/a\\ 
\end{array}
\right].\]
But this implies that
\begin{multline*}
0=a^2(a+ag+2g-2)\left\langle h_3,h_2\right\rangle-2\left\langle h_5,h_3\right\rangle\\
+2a^2\left\langle h_6,h_1\right\rangle-2a^3\left\langle h_6,h_2\right\rangle-2a\left\langle h_6,h_4\right\rangle=6(g-1),
\end{multline*}
thus we obtain a matrix having three distinct rows and columns containing a $-1$. By Theorem~\ref{T54}, $H$ belongs to the $2$-circulant family.
\end{proof}
By combining the results described in Theorems~\ref{F6CHAR}, \ref{T54} and \ref{T64}, we obtain the proof of Theorem~\ref{T1}.
\section*{Acknowledgements}
\'A.K.M.\ is supported by the Ministry of Innovation and Technology NRDI Office within the framework of the Artificial Intelligence National Laboratory, Hungary (RRF-2.3.1-21-2022-00004) and the Hungarian NRDI Office grants NKFIH K-138828 and PD-145995. F.Sz.\ is supported in part by JSPS KAKENHI Grant Number 24K06829.

\end{document}